\documentclass[11pt,oneside,reqno]{amsart}
\usepackage[a4paper,textwidth=14cm]{geometry}
\usepackage[utf8]{inputenc}
\usepackage{amsmath,amsrefs}
\usepackage{amsfonts}
\usepackage{amssymb}
\usepackage{amsthm,mathrsfs}
\usepackage{bbm}
\usepackage{color, comment}
\usepackage{enumitem}
\usepackage[hidelinks]{hyperref}

\usepackage[nameinlink]{cleveref}

\hypersetup{
    colorlinks=true,
    linkcolor=magenta,
    filecolor=magenta,      
    urlcolor=cyan,
}

\renewcommand{\H}{\mathcal{H}}
\newcommand{\N}{\mathbb{N}}
\newcommand{\R}{\mathbb{R}}
\newcommand*{\genbf}[1]{\ifmmode\mathbf{#1}\else\textbf{#1}\fi}

\newcommand{\diam}{\operatorname{diam}}

\newcommand{\gr}{\mathrm{gr}}

\renewcommand{\L}{\mathcal{L}}
\newcommand{\cH}{\mathcal H}

\newcommand{\ch}{\mathfrak h}
\newcommand{\cU}{\mathcal U}
\newcommand{\bH}{\mathbb H}
\newcommand{\sm}{\setminus}

\newcommand{\cG}{\mathcal G}
\newcommand{\res}{\mbox{\LARGE{$\llcorner$}}}

\newcommand{\beq}{\begin{equation}}
\newcommand{\eeq}{\end{equation}}

\newtheorem{theorem}{Theorem}[section]
\newtheorem{lemma}[theorem]{Lemma}
\newtheorem{proposition}[theorem]{Proposition}

\theoremstyle{definition}
\newtheorem{definition}[theorem]{Definition}
\newtheorem{remark}[theorem]{Remark}

\author{Kennedy Obinna Idu}
\address{Department of Mathematics, University of Toronto, 
			Toronto, Ontario M5S 2E4, Canada.}
\email{o.idu@utoronto.ca}

\title{On rectifiability in Heisenberg groups without a lower density condition}

\begin{document}
\begingroup
\def\uppercasenonmath#1{} 
\let\MakeUppercase\relax 
\maketitle
\endgroup

\begin{abstract}
We resolve a problem posed by Mattila, Serapioni and Serra Cassano concerning the role of density assumptions in the characterization of rectifiable sets of low codimension in Heisenberg groups. Specifically, we prove that the positive lower density condition is not required: rectifiability is completely determined by the geometric property of almost everywhere existence of approximate tangent subgroups. This provides a simplified and intrinsic criterion for rectifiable sets in the sub-Riemannian setting and sharpens the analogy with Federer’s classical Euclidean theorem.
\end{abstract}

\vspace{1cm}
 {\footnotesize

 }.

\section{Introduction}

A central problem in geometric measure theory of Heisenberg groups concerns the characterization of rectifiable sets of low codimension. A fundamental result in this direction, due to Mattila, Serapioni, and Serra Cassano \cite{MSSC10}*{Theorem~3.15}, establishes rectifiability in terms of the existence of suitable approximate tangent subgroups together with a positive lower density assumption. In that work it was explicitly asked whether the density hypothesis is essential. This question has remained open and has motivated much of the subsequent research on rectifiability in non-Euclidean settings (see, for instance, the recent survey \cite{Mat_Surv2023}).

The purpose of this article is to resolve this question in the affirmative by showing that the lower density assumption can, in fact, be removed.

\vspace{5pt}
In the sequel, $k$ denotes the topological dimension and $k_m$ the corresponding metric dimension. We write $\cG(\bH^n,k)$ for the Grassmannian of $k$-dimensional subgroups (see Definition~\ref{def.grass}). For $V \in \cG(\bH^n,k)$, let $X(p,V,\alpha)$ denote the cone with vertex $p$, axis $V$, and aperture $\alpha$ (see Definition~\ref{Def:Cone}).  
\vspace{5pt}

Our main result states:
\begin{theorem}\label{Thm:Main}
Fix $n+1 \le k \le 2n+1$. Let $E \subset \bH^n$ be $\H^{k_m}$-measurable with $\H^{k_m}(E) < \infty$. Suppose that for $\H^{k_m}$-a.e. $p \in E$ there exists $V_p \in \cG(\bH^n,k)$ such that for every $0<s<1$,
\begin{equation}\label{eq:TanConeCond}
\lim_{r \to 0^+} \frac{1}{r^{k_m}} \H^{k_m} \left( E \cap B(p,r) \setminus X(p,V_p,s) \right) =0.
\end{equation}
We call such a subgroup $V_p$ an \emph{approximate tangent subgroup} of $E$ at $p$, and write $apTan_{\bH}^k(E,p) = V_p$.
Then $E$ is $(k,\bH)-$rectifiable in the sense of Definition \ref{def:rec}.
\end{theorem}

\begin{remark}
    The converse of Theorem~\ref{Thm:Main} is established in \cite{MSSC10}*{Theorem~3.15}. We restrict to the range $n+1 \le k \le 2n+1$ because the case $1 \le k \le n$ was already established in \cite{MSSC10}*{Theorem~3.14}. 
\end{remark}

\subsection{Background and review of result}
Theorem~\ref{Thm:Main} fits into a line of characterizations of rectifiable sets in terms of an intrinsic geometric criterion. 
In the Euclidean setting, \cite{Mat95}*{Theorem~15.19} establishes that rectifiability can be characterized by the existence of \textit{approximate tangent planes}. 
In Heisenberg groups, an analogue for low-dimensional subsets was obtained in \cite{MSSC10}*{Theorem~3.14}, where rectifiability is described in terms of the existence of suitable approximate tangent subgroups. 
More recently, in \cite{IMM20}*{Theorem~1.1}, a version was proved in the general context of homogeneous spaces.

Two main strategies are typically employed. The first is a structure-theoretic approach, based on a geometric lemma asserting that sets locally contained in cones are rectifiable, combined with a density estimate for purely unrectifiable sets. Under the approximate tangent plane condition, this argument shows that the purely unrectifiable component is negligible. This strategy is used in several places including \cite{IMM20}*{Theorem~1.1}, \cite{Mat95}*{Theorem~15.19} and \cite{Mat2022}*{Theorem~3.8} with the density estimate, modeled on Federer's  \cite{Fed69}*{Lemma~3.3.6}, as its critical ingredient. Its proof exploits the linear structure and Lipschitz continuity of horizontal projections.

The second strategy is a standard  decomposition argument using density conditions. One first establishes that the lower density is positive almost everywhere on a set with approximate tangent planes. This approach is adopted in \cite{MSSC10}*{Theorem~3.14} where they take advantage of the identification of horizontal subgroups with Euclidean subspaces, the linearity of horizontal projections and the decompositions of dyadic cubes.
\subsection{Motivations and strategy}
The proof of Theorem~\ref{Thm:Main} follows a structure-theoretic approach, combining a geometric lemma with a density estimate. While these elements are standard in low-dimensional or Euclidean settings, their direct application fails in the present context.

Standard arguments rely critically on the Lipschitz and homomorphism properties of horizontal projections, which allow sets contained in cones to be covered by pieces of Lipschitz graphs. In the current setting, these properties are lost for vertical projections, rendering the classical geometric lemma inapplicable. Similarly, standard density estimates for purely unrectifiable sets depend on three components: a covering argument for horizontal projections of an essential set, a $k_m$-dimensional Hausdorff measure estimate over the covering indices, and the ability to cover a cylinder with two cones. Each of these steps fails when vertical projections are used due to non-adaptability of the covering argument, dimension distortion, and the collapse of computations involving intrinsic cylinders defined using vertical projections.

To overcome these challenges, in the geometric lemma we use the cone characterization of intrinsic Lipschitz graphs (see, e.g., \cites{FSSC6, FranchiSerapioni2016IntrLip}) together with their recent identification with $\bH$-regular surfaces in \cite{Vittone2022}. This allows us to cover sets that satisfy the cone condition by pieces of $\bH$-regular surfaces. For the density estimate, avoiding vertical projections requires addressing pathologies of the left-invariant metric, including the non-isometry of inversion and $\frac{1}{2}$-H\"older behavior at small scales. Using small-scale analysis of sets and geometric distance functionals, together with a control estimate for distances from $k$-subgroups under inversion, we can cover cylinders with two cones and obtain the necessary Hausdorff measure estimates. This technique of controlling inversions within cones and cylinders uses the symplectic structure of Heisenberg groups (specifically, the bilinear form $\omega$ appearing in the group operation) and is central to our proof.

\subsection{Structure of paper}
The structure of the paper is as follows. In Section~\ref{sec:preliminaries} we recall some basic definitions and facts, including the geometry of the Heisenberg group, intrinsic Lipschitz graphs, rectifiable sets, and approximate tangent subgroups. Section~\ref{sec:density} contains the technical result about a density estimate on purely unrectifiable sets needed to overcome the absence of a lower density hypothesis. The main theorem is proved in Section~\ref{sec:main}, where we establish that the existence of approximate tangent subgroups alone suffices to guarantee rectifiability. Finally, in Section~\ref{sec:applications} we present a consequence of the result in connection with higher-order rectifiability in Heisenberg groups.

\addtocontents{toc}{\protect\setcounter{tocdepth}{1}}
\subsection*{Acknowledgements.} 
The author is grateful to Francesco Paolo Maiale for many stimulating discussions, to Valentino Magnani for introducing him to the problem and motivation, and to Pertti Mattila for comments on an earlier draft that prompted a complete revision of techniques. He also thanks Almut Burchard, Antoine Julia, and Davide Vittone for valuable feedback on earlier drafts. This work began during his tenure as an AIMS–Fields–Perimeter postdoctoral fellow at the Fields Institute, Toronto, whose support and stimulating environment are gratefully acknowledged, and was completed in part during a visit to the ICTP–East African Institute for Fundamental Research at the University of Rwanda. Finally, he thanks the reviewers for their valuable comments and suggestions.
\addtocontents{toc}{\protect\setcounter{tocdepth}{2}}

\section{Preliminaries}\label{sec:preliminaries}
\subsection{The Heisenberg groups}
The Heisenberg group $\bH^n$ is a simply-connected Lie group whose Lie algebra $\ch^n$ has a step two stratification:
$$
\ch^{n}=\ch_{1}^{} \oplus \ch_{2}^{}
$$
where $\ch_{1}=\operatorname{span}\left\{X_{1}, \ldots, X_{n}, Y_{1}, \ldots, Y_{n}\right\}$ and $\ch_{2}=\operatorname{span}\{T\}$ with only nontrivial commutator relation
$\left[X_{j}, Y_{j}\right]=-4T$, for all $1\le j\le n$. The vector fields $X_{1},\ldots,X_{n},Y_{1},\ldots,Y_{n}$ define the \emph{horizontal vector bundle} $\mathrm{H\bH}^{n}$, a subbundle of the tangent bundle $\mathrm{T\bH}^{n}$.  Via exponential coordinates $\bH^n$ can be identified with $\R^{2n+1}$, and the group law by the Baker-Campbell-Hausdorff formula can be written as
$$
p \cdot q:=\left(p^{\prime}+q^{\prime}, p_{2 n+1}+q_{2 n+1}+2 \omega\left(p',q'\right)\right),
$$
where $p':=\left(p_{1}, \cdots, p_{2 n}\right)$ and $\omega (p',q'):=\sum_{i=1}^{n}\left(p_{i+n} q_{i}-p_{i} q_{i+n}\right)$. The inverse of $p$ is $p^{-1}:=\left(-p^{\prime},-p_{2 n+1}\right)$ and $e=0$ is the identity of $\mathbb{H}^{n}$. The centre of $\mathbb{H}^{n}$ is the subgroup $\mathbb{T}:=\left\{p=\left(0, \ldots, 0, p_{2 n+1}\right)\right\}$. For any $q \in \mathbb{H}^{n}$ and $r>0,$ we denote as $\tau_{q}: \mathbb{H}^{n} \rightarrow \mathbb{H}^{n}$ the left translation $p \mapsto q \cdot p=\tau_{q}(p)$ and as $\delta_{r}: \mathbb{H}^{n} \rightarrow \mathbb{H}^{n}$ the dilation
$$
p \mapsto\left(r p^{\prime}, r^{2} p_{2 n+1}\right)=\delta_{r} p.
$$
We denote, for all $p, q \in \mathbb{H}^{n}$
$$
\|p\|:=d(p, e):=\max \left\{\left\|p'\right\|_{\mathbb{R}^{2 n}},\left|p_{2 n+1}\right|^{1 / 2}\right\}
$$
and
\begin{equation}\label{Def:HMetric}
d(p, q)=d\left(q^{-1} \cdot p, e\right)=\left\|q^{-1} \cdot p\right\|
\end{equation}
Then, for all $p, q, z \in \mathbb{H}^{n}$ and for all $r>0$
$$
d(z \cdot p, z \cdot q)=d(p, q) \quad \text { and } \quad d\left(\delta_{r} p, \delta_{r} q\right)=r d(p, q)
$$

The Lie algebra $\mathfrak{h}^n$ is endowed with the scalar product $\langle\cdot, \cdot\rangle,$ that makes $X_{1}, \ldots, Y_{n}, T$ orthonormal, and induces the norm $\|\cdot\|$.

\subsection{The intrinsic Grassmannian}

A subgroup $S \subset \bH^n$ is a {\em homogeneous subgroup} if $\delta_r(S) \subseteq S$ for all $r > 0$.
A homogeneous subgroup $S$ is either {\em horizontal}, i.e. contained in $\mathrm{exp} (\mathfrak{h}_1)$, or {\em vertical}, i.e., containing the center $\mathbb{T}$ of $\bH^n$. We denote
\[
d(p,S) := \inf_{s \in S} d(p,s) = \inf_{s \in S} \| p^{-1}s\|.
\] 
In $\bH^n$, vertical subgroups are normal but horizontal subgroups are abelian.

\begin{definition}
Homogeneous subgroups $V$ and $W$ of $\bH^n$ are {\em complementary} if $ V \cap W = \{0\}$ and $\bH^n = W \cdot V$. If $W$ is normal then $\bH^n$ is a {\em semidirect} product, $\bH^n = W \rtimes V$. We define projections $\pi_V : \bH^n \to V$ and $\pi_W : \bH^n \to W$ such that
\[
\mathrm{id}_{\bH^n} = \pi_W \cdot \pi_V.
\]
The following algebraic equalities hold:
\[ \begin{aligned}
& \pi_W(p^{-1}) = \pi_V^{-1}(p) \cdot \pi_W^{-1}(p) \cdot \pi_V(p), \quad \pi_V(p^{-1}) = \pi_V^{-1}(p), \\
& \pi_W(\delta_\lambda p) = \delta_\lambda \pi_W(p), \quad \pi_V(\delta_\lambda p) = \delta_\lambda \pi_V(p), \\
& \pi_W(p \cdot q) = \pi_W(p) \cdot \pi_V(p) \cdot \pi_W(q) \cdot \pi_V^{-1}(p), \quad \pi_V(p\cdot q) = \pi_V(p) \cdot \pi_V(q).
\end{aligned}\]
\end{definition}

\begin{proposition}\cite{MSSC10} \label{prop:proj-estimates}
If $\bH^n = W \rtimes V$, then $\pi_V$ and $\pi_W$ are continuous, $\pi_V$ is a $h$-homomorphism and there is $c:=c(V,W)> 0$ such that for all $p \in \bH^n$
\[
\begin{aligned}
& c \|\pi_V(p)\| \leq d(p, W) \leq \|\pi_V(p)\|,
\\ & c \|\pi_V^{-1}(p) \cdot \pi_{W}(p) \cdot \pi_V(p)\| \leq d(p, V) \leq \|\pi_V^{-1}(p) \cdot \pi_{W}(p) \cdot \pi_V(p)\|. \end{aligned}
\]
\end{proposition}

We now introduce the notion of {\em intrinsic Grassmannian} as in \cite{MSSC10}.

\begin{definition}\label{def.grass}
A $k$-subgroup $V$ belongs to the $k$-Grassmannian $\cG(\bH^n,k)$ if there exists a $(2n+1-k)$-subgroup $W$ such that $\bH^n = W \cdot V$. Moreover, the union
\[
\cG(\bH^n) := \bigcup_{k = 0}^{2n+1} \cG(\bH^n,k)
\]
is often referred to as the {\em intrinsic Grassmannian} of $\bH^n$.
\end{definition}

\begin{proposition}\label{prop.str}
The trivial subgroups $\{e\}$ and $\bH^n$ are the unique elements of $\cG(\bH^n,0)$ and $\cG(\bH^n,2n+1)$ respectively and \mbox{}
\begin{enumerate}
\item for $1 \le k \le n$, $\cG(\bH^n,k)$ coincides with the set of all horizontal $k$-subgroups;
\item for $n+1 \le k \le 2n$, $\cG(\bH^n,k)$ coincides with the set of all vertical $k$-subgroups.
\end{enumerate}
Furthermore, any vertical subgroup $W$ with linear dimension in $\{1,\ldots,n\}$ is not an element of the intrinsic Grassmannian of $\bH^n$.
\end{proposition}

\begin{remark} 
The Grassmannian $\cG(\bH^n)$ is a subset of the Euclidean analogue and has the same topology. Moreover, $\cG(\bH^n,k)$ is compact with respect to the metric
\[
\rho(S_1, S_2):= \max_{\|x \|=1} d\left( \pi_{S_1}(x),\pi_{S_2}(x)\right).
\]
\end{remark}

\begin{lemma}\label{Eq:rho}
Let $1\le k\le 2n+1$. For any $S_1,S_2\in \cG(\bH^n,k)$,
\[
\rho(S_1, S_2)=\rho(S_1^\perp, S_2^\perp),
\]
where $S^\perp$ denotes the metric orthogonal complement of $S\in \cG(\bH^n,k)$ in the sense of the intrinsic Grassmannian.
\end{lemma}
\begin{proof} 
The proof is an elementary exercise, we write for completeness.\\
$(i)$ Case $1\le k\le n$: Let $S_1,S_2\in \cG(\bH^n,k)$
\begin{align*}
\rho(S_1,S_2)&=\max_{\|p\|=1}\|\pi_{S_1}(p)^{-1}\pi_{S_2}(p)\|=\max_{\|p\|=1}\|\pi_{S_1}(p)^{-1}(p^{-1})^{-1}p^{-1}\pi_{S_2}(p)\|\\
&=\max_{\|p\|=1}\|\pi_{S_1}(p^{-1})\pi_{S_1}^{-1}(p^{-1})\pi_{S_1^\perp}^{-1}(p^{-1})\pi_{S_2^\perp}(p^{-1})\pi_{S_2}(p^{-1})\pi_{S_2}^{-1}(p^{-1})\|\\
&=\max_{\|p^{-1}\|=1}\|\pi_{S_1^\perp}^{-1}(p^{-1})\cdot\pi_{S_2^\perp}(p^{-1})\|=\rho(S_1^\perp,S_2^\perp).
\end{align*}

\noindent
$(ii)$ Case $n+1\le k\le 2n+1$: Let $S_1,S_2\in \cG(\bH^n,k)$
\begin{align*}
\rho(S_1,S_2)&=\max_{\|p\|=1}\|\pi_{S_1}(p)^{-1}\pi_{S_2}(p)\|=\max_{\|p\|=1}\|\pi_{S_1}(p)^{-1}pp^{-1}\pi_{S_2}(p)\|\\
&=\max_{\|p\|=1}\|\pi_{S_1^\perp}(p)\cdot\pi_{S_2^\perp}(p)^{-1}\|\\
&=\max_{\|p^{-1}\|=1}\|\pi_{S_1^\perp}^{-1}(p^{-1})\cdot\pi_{S_2^\perp}(p^{-1})\|=\rho(S_1^\perp,S_2^\perp).\qedhere
\end{align*}
\end{proof}

\begin{definition}[Intrinsic cones]\label{Def:Cone}
Let $S$ be a $k$-subgroup of $\bH^n$ and $\alpha\in (0,1)$. The intrinsic cone $X(p_0,S,\alpha)$ centered at $p_0\in \bH^n$, axis $S$ and aperture $\alpha$ is given by
\[
X(p_0,S,\alpha):=\left\{p\in \bH^n:d(p_0^{-1}\cdot p,S)\le \alpha d(p,p_0)\right\}.
\]
We set $X(p_0,r,S,\alpha)=X(p_0,S,\alpha)\cap B(p_0,r)$ for any $r>0$. We further denote the $\rho$-neighbourhood of the $k$-subgroup $S$ by
\[
N(S,\rho)=\{p\in \bH^n: d(p,S)\le \rho\}.
\]
\end{definition}

\subsection{Intrinsic Lipschitz graphs}
We fix a splitting $\bH^n=W\cdot V$ in terms of homogeneous complementary subgroups $W,V$.

Given $A\subset W$ and a map $\phi:A\to V$, the intrinsic graph of $\phi$ is the set
\[
\gr_\phi:=\{w\phi(w):w\in A\}\subset\bH^n.
\]

 \begin{definition}\label{def:grafLip}
Let $A\subset W$; we say that a map $\phi:A\to V$ is {\em intrinsic Lipschitz continuous} if there exists $\alpha>0$ such that
\begin{equation}\label{eq:defLip}
\forall\:p\in\gr_\phi\qquad \gr_\phi\cap X(p,V,\alpha)=\{p\}.
\end{equation}
where $X(p,V,\alpha)$ denotes the cone with vertex $p$, axis $V$, and aperture $\alpha$.  
The \emph{intrinsic Lipschitz constant} of $\phi$ is
\[
\mathrm{Lip}_H(\phi) := \inf \left\{ \alpha > 0 \, : \, \eqref{eq:defLip} \text{ holds} \right\}.
\]
\end{definition}

\subsection{Hausdorff measure and densities}
We define the {\em $k$-dimensional Hausdorff measure} $\cH^k$ of a set $E\subset \bH^n$ by 
\(
\cH_d^k(E):=\sup_{\delta>0}\cH_{\delta}^k(E),
\)
where 
$$\cH_{\delta}^k(E)=\inf \left\{ \sum_i 2^{-k}\diam(E_i)^k: E\subset \bigcup_i E_i, \diam(E_i)\le \delta\right\}.$$
Given a $\cH^k$-measurable subset $E \subset \bH^n$ we define the corresponding upper and lower $k$-densities of $E$ at $p\in \bH^n$ as follows:\[
\Theta^{\ast k}(E, p)= \limsup_{r \to 0} \frac{\cH^k(E\cap B(p, r))}{r^k} \quad\text{and} \quad \Theta_\ast^k(E, p)= \liminf_{r \to 0} \frac{\cH^k(E\cap B(p, r))}{r^k}.
\]
\noindent
We recall the standard density estimates (see \cite{Fed69}*{Section~2.10.19} or \cite{Mat95}*{Theorem~6.2}:
\begin{lemma} \label{lemma:densities}
Let $E \subset \bH^n$ be $\cH^k$-measurable with $\cH^k(E)<+\infty$. Then 
\begin{enumerate}
\item For $\cH^k$-a.e. $p \in E$, it turns out that $2^{-k} \le \Theta^{\ast k}(E, p) \le 1$.
\item For $\cH^k$-a.e. $p \in\bH^n\sm E$, it turns out that $\Theta^{\ast k}(E, p) = 0$.
\end{enumerate}
\end{lemma}

\subsection{\texorpdfstring{$\bH$-} -regular surfaces}
Though our main results concern Heisenberg groups, we introduce a notion of differentiability for maps between general \textit{Carnot groups}. We refer to the monograph \cite{BLU2007stratified} for a thorough introduction to Carnot groups.
\begin{definition}[Pansu differentiability]
Let $\left(\mathbb{G}_{1}, d_{1}\right)$ and $\left(\mathbb{G}_{2}, d_{2}\right)$ be Carnot groups and $\Omega \subset \mathbb{G}_{1}$. We say that a function $f: \Omega \to \mathbb{G}_{2}$ is {\em Pansu differentiable} at $p \in \Omega$ if there is a homogeneous homomorphism $L_{p}: \mathbb{G}_{1} \to \mathbb{G}_{2}$ such that
\[
\frac{d_{2}\left(f(p)^{-1} \cdot f\left(q \right), L_{p}\left(p^{-1} \cdot q \right)\right)}{d_{1}\left(p, q\right)} \to 0 \qquad \text { as } d_{1}\left(p, q \right) \to 0 \text{ and }  q \in \Omega.
\]
The homogeneous homomorphism $L_{p}$ is called the \textit{Pansu differential} of $f$ at $p$ and is usually denoted by $d_{H} f_{p}$.
\end{definition}

\begin{definition}\label{Def:RegSurf}
A subset $S\subset\mathbb{H}^{n}$ is a $k$-dimensional $\bH$-regular surface if:
\begin{itemize}
\item[(i)]$1 \leq k \leq n$ and
for any $p \in S$ there are open sets $\mathcal{U} \subset \mathbb{H}^{n}, \mathcal{V} \subset \mathbb{R}^{k}$ and a function $\varphi: \mathcal{V} \rightarrow \mathcal{U}$
such that $p \in \mathcal{U}, \varphi$ is injective, and continuously Pansu differentiable with $d_{H} \varphi_{p}$ injective, and
$$
S \cap \mathcal{U}=\varphi(\mathcal{V})
$$
\item[(ii)]$n+1 \leq k \leq 2 n$
and for any $p \in S$ there are an open set $\mathcal{U} \subset \mathbb{H}^{n}$ and a function $f: \mathcal{U} \rightarrow \mathbb{R}^{2 n+1-k}$ such that $p \in \mathcal{U}, f \in\left[\mathbf{C}_{H}^{\mathbf{1}}(\mathcal{U})\right]^{2 n+1-k}$, with $d_{H} f_{q}$ surjective for every $q \in \mathcal{U}$
and
$$
S \cap \mathcal{U}=\{q \in \mathcal{U}: f(q)=0\}
$$ %
\end{itemize}
\end{definition}

\begin{remark} \label{Rem:codimReg} It follows from the definition above that a set $S$ is a $k$-dimensional $\bH$-regular surface if and only if $S$ is locally the intersection of $(2n+1-k)$ $1$-codimensional $\bH$-regular surfaces with linearly independent normal vectors.
\end{remark}

To conclude this introductory section, we give a formal definition of rectifiability for a subset of the Heisenberg group (see e.g., \cite{MSSC10}).

\begin{definition} \label{def:rec}
A measurable set $E \subset \bH^n$ is $(k,\bH)-$rectifiable if there exist {} $k$-dimensional {} $\bH$-regular surfaces $S_i$, with $i \in \N$, such that
\[
\cH^{k_m} \left(E\sm\bigcup_{i\in\N}S_i\right)=0,
\]
where {} $k_m=k$ if {} $1\le k \le n$ and $k_m=k+1$ if {} $n+1\le k \le 2n+1$.

A measurable set $S\subset\bH^n$ is purely $(k, \bH)$-unrectifiable if $\H^{k_m}(S\cap E)=0$ for every $(k, \bH)$-rectifiable set $E\subset\bH^n$.
\end{definition}

\section{Density estimate on purely unrectifiable sets}\label{sec:density}
This section establishes a quantitative decay estimate for intersections of cones with purely \(k\)-unrectifiable sets. The main result, Theorem~\ref{thm:density-estimate}, provides a uniform bound on such intersections at small scales, adapting classical Euclidean density estimates (see, e.g., \cite{Fed69}*{Lemma~3.3.6}, \cite{Mat95}*{Lemma~15.14}) to the Heisenberg setting.

We first establish the following Heisenberg analogue of the classical ``geometric lemma'' 
in Euclidean rectifiability theory (see, e.g., \cite{Mat95}*{Lemma ~15.13}). 
In the Euclidean setting, a set contained in a cone can be
covered by pieces of Lipschitz graphs, and hence shown to be rectifiable. 
The lemma harnesses the equivalent role of intrinsic cones and intrinsic Lipschitz graphs in our setting. Furthermore, it provides a connection between cone-type 
criteria and rectifiability, and a crucial tool in the proof of the density
estimate in Theorem~\ref{thm:density-estimate}.

\begin{lemma}[Geometric lemma]\label{lem:geom}
Fix $n+1 \le k \le 2n+1$. Suppose $E \subset \bH^n$ is a $\H^{k_m}$-measurable set with $\H^{k_m}(E) < \infty$, $V\in \cG(\bH^n,k)$, $r>0$ and $0<s<1$. If
\begin{equation}\label{Eq:geom_lemma}
E\cap B(p,r)\cap X(p,V^\perp,s)\subset \{p\},
\end{equation}
whenever $p\in E$. Then $E$ is locally an intrinsic Lipschitz graph over $V$ along $V^\perp$. In particular, $E$ is $(k,\bH)$-rectifiable.
\end{lemma}
\begin{proof}
Let $p\in E$ be fixed. Without loss of generality (up to a left translation), suppose $p=e$.
The inclusion hypothesis \eqref{Eq:geom_lemma} gives the cone criterion for intrinsic Lipschitz graphs which implies there exists $0<r_p\le r$ such that $E\cap B(p,r_p)$ is an intrinsic Lipschitz graph over $V$ along $V^\perp$ (see, e.g., \cite{FranchiSerapioni2016IntrLip}*{\S 2.2.3}). Since $\mathbb{H}^n$ is separable, the cover $\{B(p, r_p) : p \in E\}$ has a countable subcover.
Each intrinsic Lipschitz graph is $(k, \mathbb{H})$-rectifiable by \cite{Vittone2022}*{Corollary 7.4}, and countable unions of rectifiable sets are rectifiable. In particular, $E$ is $(k,\bH)$-rectifiable.\qedhere
\end{proof}

The next lemma provides a sharp quantitative estimate showing that if a point lies within a cone of controlled aperture around a horizontal subgroup, then its inverse also lies within a cone of comparable aperture, with the constants depending only on the original parameters.
\begin{lemma}\label{ineq:cone_inversion}
    Let $V\in \cG(\bH^n,k)$ with $n+1\le k\le 2n+1$. Suppose $\alpha, \beta,M>0$ and $0<s\le 1$. If for any $x\in \bH^n$ with $\|x\|\le \alpha^2 M$ we have
    \[
    d(x,V^\perp)\le s^2\beta^2 M,
    \]
    then
    \[
    d(x^{-1},V^\perp)\le (\beta^2+2\sqrt{2}\alpha\beta)sM
    \]
\end{lemma}
\begin{proof}
Let $x=(x',t)\in \bH^n$ with $\|x\|\le\alpha^2 M$. Pick $v\in V^\perp$ such that $d(x,v):=d(x,V^\perp)$. Then 
\[
\|v\|-\|x\|\le d(x,V^\perp) \le \|x\|
\implies
\|v\|\le 2\|x\|\le 2\alpha^2 M. 
\]
The Heisenberg distance is
\[
d(x,v)=\max\left\{|x'-v|, |t + 2\omega(x',v)|^{1/2}\right\}
\]
It follows from the assumption that 
\[|x'-v|\le s^2\beta^2 M,\quad  |t + 2\omega(x',v)|^{1/2}\le s^2\beta^2 M.
\]

\noindent
To estimate the distance of $x^{-1}$ from $V^\perp$, consider
\[
d(x^{-1}, V^\perp)=\inf_{u\in V^\perp}\max\left\{|x'+u|, |t+2\omega(x',u)|^{1/2}\right\}
    \le \max\left\{|x'-v|, |t-2\omega(x',v)|^{1/2}\right\}.
\]
For the second term, we have
\[
|t - 2 \omega(x', v)|^{1/2} = |t + 2 \omega(x', v) - 4 \omega(x', v)|^{1/2} \le \big( |t + 2 \omega(x', v)| + 4 |\omega(x', v)| \big)^{1/2}.
\]
Using the subadditivity of the square root and previous bounds,
\[
|t - 2 \omega(x', v)|^{1/2} \le |t + 2 \omega(x', v)|^{1/2} + 2 \sqrt{|x' - v| \cdot |v|} \le s^2 \beta^2 M + 2 \sqrt{2} \alpha \beta s M.
\]
Finally, since $|x'-v| \le s^2 \beta^2 M \le (\beta^2 + 2\sqrt{2}\,\alpha \beta) s M$, we have
\[
d(x^{-1},V^\perp)\le \max\left\{|x'-v|, (\beta^2+2\sqrt{2}\alpha\beta)sM\right\}=(\beta^2+2\sqrt{2}\alpha\beta)sM,
\]
concluding on the desired estimate.
\end{proof}

\begin{theorem}[Density estimate]\label{thm:density-estimate}
Fix $n+1\le k\le 2n+1$. Let $E\subset\bH^n$ be a purely $(k, \bH)$-unrectifiable set and $V \in \cG(\bH^n,k)$. Let $\lambda > 0$, $0 < s < 1$, and $0<\delta<1$. Suppose for all $p \in E$ and every $0 < r \leq \delta$ we have
\begin{equation}\label{eq.assumpt.1}
\cH^{k_m}\left(E\cap X(p,r,V^\perp,s)\right) \leq \lambda r^{k_m} s^{k_m}.
\end{equation}
Then for any $x_0\in \bH^n$,
\begin{equation}\label{eq:densityest}
\Theta^{\ast k_m}(E,x_0) \leq ( 2^{11k_m+1}\cdot 2000^{4k_m} s^{-7k_m})\lambda.
\end{equation}
\end{theorem}

\begin{proof}
Let $x_0\in \bH^n$ and $0<\rho\le s\frac{\delta}{200}$. Denote by $\Bar{s}:=s^2/100$ and let 
\[F:=E\cap B(x_0,\rho).\]
\textbf{Step 1: Considering an essential set}.
Using the fact that $E$ is purely $(k,\bH)$-unrectifiable, by Lemma \ref{lem:geom} it suffices to consider
\begin{equation*}
C:=F\cap\{x: (F\setminus\{x\})\cap X(x,V^\perp,\bar{s})\ne \emptyset\}
\end{equation*}
 since $\H^{k_m}(F\setminus C)=0$.
For each $x\in C$  define,
\[h(x)=\sup\left\{\|x^{-1}y\|: y\in F\cap X(x,V^\perp,\bar{s})\right\}\]
and observe that $0< h(x)\le \diam(F)\le 2\rho$.

\vspace{3mm}
\noindent
\textbf{Step 2: Covering argument.} We show a covering result which uses width enlargement of compact cylinders. We define $\bar{s}_1:=\bar{s}/20$. It is easy to check that
\begin{equation}\label{incl:covering}
    C\subset \bigcup_{x\in C}\left[N\left(xV^\perp,\bar{s}_1^4h(x)/36\right)\cap B(x,\rho/4)\right]\subset B\left(x_0,2\rho\right).
\end{equation}

\noindent
\textbf{Step 2a: Countable disjoint family of cylinders}.
For any $q\in C$, we denote $\bar{h}(q):=h(q)/36$ and $A_q:=N\left(qV^\perp,\bar{s}_1^4\bar{h}(q)\right)\cap B(q,\rho/4)$. Using \cite{Fed69}*{Theorem~2.8.4} we obtain a disjoint subfamily $\left\{A_x:x\in D\subset C\right\}$ (which is countable since $\bH^n$ is a separable metric space) such that for every $x\in C$ there exists $y\in D$ with $A_x\cap A_y\ne \emptyset$ and $\bar{h}(x)\le 2\bar{h}(y)$. Moreover, if for any $q\in C$, we define
    \[
    \bar{A}_q:=\bigcup \{A_u: u\in C, A_u\cap A_q\ne \emptyset, \bar{h}(u)\le 2\bar{h}(q)\},
    \]
    then, we have (see \cite{Fed69}*{Corollary~2.8.5}),
    \begin{equation}\label{incl:enlarged_covering}
        \bigcup_{q\in C} A_q \subset \bigcup_{\bar{q}\in D} \bar{A}_{\bar{q}}.
    \end{equation}

\noindent
\textbf{Step 2b: Control Hausdorff estimate}.
From \eqref{incl:covering}, {it follows} from estimating the Lebesgue measure $\L^{2n+1}$ of disjoint bounded cylinders, that
\[
\L^{2n+1}\left( \bigcup_{x\in D} A_x\right)=\sum\limits_{x\in D}^{}\left(\frac{\bar{s}_1
^4 {h}(x)}{36}\right)^{k_m}\left(\frac{\rho}{4}\right)^{2n+2-k_m}\le 2^{2n+2}{\rho}^{2n+2}.
\]
\noindent
We obtain the estimate
\begin{equation}\label{ineq:sum-estimate}
\sum\limits_{x\in D}^{}\left(\bar{s}_1^4 {h}(x)\right)^{k_m}\le 2^{10k_m}\rho^{k_m}.
\end{equation}

\noindent
\textbf{Step 2c: Considering intersected family}. We consider the intersected family $\left\{A_x\cap C: x\in C\right\}$. For any $q\in C$, denote 
    \[\tilde{A}_q:=\bigcup \{A_u\cap C: u\in C, A_u\cap A_q\cap C\ne \emptyset, \bar{h}(u)\le 2\bar{h}(q)\}.\] 
    It follows from \eqref{incl:enlarged_covering} that 
    \begin{equation}\label{incl:tilde_covering}
        \bigcup_{q\in C} A_q \cap C\subset \bigcup_{\bar{q}\in D} \tilde{A}_{\bar{q}}.
    \end{equation}
\textbf{Claim}: For every $y\in C$, passing from $\bar{s}_1$ to the larger aperture $\bar{s}$, we have
\begin{equation}\label{Claim:enlarged_cylinder}
    \tilde{A}_y\subset C\cap N\left(yV^\perp,\bar{s}h(y)\right)\cap B(y,\rho).
\end{equation}
    \noindent
    Let $y\in C$ and let $z_1\in \tilde{A}_y$. We find $x\in C$ such that $z_1\in A_x\cap C,\, A_x\cap A_y\cap C\ne\emptyset $, and  $\bar{h}(x)\le 2\bar{h}(y)$. Let $z_2\in A_x\cap A_y\cap C$. Then
    \[
    d(z_1,y)\le d(z_1,z_2) + d(z_2,y)\le \rho/2 + \rho/4 < \rho.
    \]
\noindent
    Moreover, we always have that 
    \begin{equation}
        \|x^{-1}z_1\| \le h(x) \text{ and }\|x^{-1}z_2\| \le h(x).
    \end{equation}
Indeed, for each $i\in\{1,2\}$, If $z_i\in X\!\left(x,V^\perp,\bar{s}\right)$ then $\|x^{-1}z_i\|\le h(x)$, otherwise, using that $z_i\in A_x$,
\[
\|x^{-1}z_i\|<\frac{1}{\bar{s}_1}d\left(x^{-1}z_i,V^\perp\right)\le \bar{s}_1^3 \bar{h}(x)\le h(x). 
\]
\noindent    
Let $v_1\in V^\perp$ such that $d(z_1,xv_1):= d(z_1,xV^\perp)$. Then 
\[\|v_1\|-\|x^{-1}z_1\|\le d(x^{-1}z_1,V^\perp) \le \|x^{-1}z_1\|\]
which implies that
\[\|v_1\|\le 2\|x^{-1}z_1\|\le 2h(x). 
\]
We first compute:
\begin{eqnarray*}
d(z_1,z_2V^\perp)&=&d(z_2^{-1}x\cdot x^{-1}z_1,V^\perp)\le d(x^{-1}z_1,v_1) + d(z_2^{-1}x\cdot v_1,V^\perp)\\
&=& d(x^{-1}z_1,V^\perp) + {d(z_2^{-1}x\cdot v_1,V^\perp)}\\
&\le& \bar{s}_1^4\bar{h}(x) + \underline{d(z_2^{-1}x\cdot v_1,V^\perp)}
\end{eqnarray*}
\noindent
To evaluate the underlined: Denote $\bar{x}=z_2^{-1}x\cdot v_1$. Then
\[
\|\bar{x}\|\le \|x^{-1}z_2\| + \|v_1\|\le 3h(x). 
\]
Using $z_2\in A_x$, we have
\[
d(\bar{x}^{-1}, V^{\perp})\le d(x^{-1}z_2,V^\perp)\le \bar{s}_1^4\bar{h}(x)
\]
Applying Lemma \ref{ineq:cone_inversion} with $\alpha=2, \, \beta=1/6$ and $M=h(x)$
we have the following:
\[
d(z_2^{-1}x\cdot v_1,V^\perp)=d(\bar{x},V^\perp)\le 2\bar{s}_1^2{h}(x)
\]
Hence,
\begin{equation}\label{eqn:cone_elements}
    d(z_1,z_2V^\perp)\le \bar{s}_1^4\bar{h}(x) + 2\bar{s}_1^2{h}(x) \le 3\bar{s}_1^2{h}(x).
\end{equation}
From our hypothesis, we have $d(z_2,yV^\perp)\le \bar{s}_1^4\bar{h}(y)$. A Similar argument as above gives $\|y^{-1}z_2\|\le h(y)$. 
Applying Lemma \ref{ineq:cone_inversion} with $\alpha=1, \, \beta=1/6$ and $M=h(y)$
gives
\begin{equation}\label{eqn:int_inv}
    d(y,z_2V^\perp)\le 2\bar{s}_1^2h(y).
\end{equation}

\vspace{3mm}
\noindent
Let $v_2\in V^\perp$ such that $d(z_2^{-1}z_1,v_2):=d(z_2^{-1}z_1,V^\perp)$. We obtain
\[
\|v_2\|\le 2\|z_2^{-1}z_1\|\le 2(\|z_2^{-1}x\|+\|x^{-1}z_1\|)\le 4h(x)\le 8h(y)
\]
We have, using \eqref{eqn:cone_elements}, that
\begin{eqnarray}\nonumber
d(z_1,yV^\perp)&=&d(y^{-1}z_2\cdot z_2^{-1}z_1,V^\perp)\le d(z_2^{-1}z_1,v_2) + d(y^{-1}z_2\cdot v_2,V^\perp)\\\nonumber
&=& d(z_2^{-1}z_1,V^\perp) + {d(y^{-1}z_2\cdot v_2,V^\perp)}\\
&\le& 3\bar{s}_1^2{h}(x) + \underline{d(y^{-1}z_2\cdot v_2,V^\perp)}.\label{ineq:underlined2}
\end{eqnarray}
\noindent
Again, we proceed to evaluate the underlined: Denote $\bar{y}_2=y^{-1}z_2\cdot v_2$. Then
\[
\|\bar{y}_2\|\le \|z_2^{-1}y\| + \|v_2\|\le h(y)+8h(y)=9h(y)
\]
We have from \eqref{eqn:int_inv} that
\[
d(\bar{y}_2^{-1}, V^{\perp})\le d(z_2^{-1}y,V^\perp)\le 2\bar{s}_1^2h(y)
\]
Applying Lemma \ref{ineq:cone_inversion} with $\alpha=3, \, \beta=\sqrt{2}$ and $M=h(y)$
we have
\[
d(y^{-1}z_2\cdot v_2,V^\perp)=d(\bar{y}_2,V^\perp)\le 14\bar{s}_1{h}(y)
\]
Hence, from \eqref{ineq:underlined2}, using that $h(x)\le 2h(y)$ and $\bar{s}_1=\bar{s}/20$, we have 
\begin{equation}
    d(z_1,yV^\perp)\le 3\bar{s}_1^2{h}(x) + 14\bar{s}_1{h}(y) \le 20\bar{s}_1{h}(y)\le \bar{s}h(y)
\end{equation}
The claim \eqref{Claim:enlarged_cylinder} is established.

\vspace{3mm}
\noindent
\textbf{Step 3: Covering a cylinder with two cones.}
Fix $x\!\in\! D$ and choose $y\!\in\! F\cap X\!(x,V^\perp,\bar{s})$ such that
$\|x^{-1}y\|\!>\!\frac{9}{10}h(x)$.
We claim that passing from $\bar{s}$ to the larger scale $s$, we have
\begin{equation}\label{eq:2cones}
C\cap N\left(xV^\perp,\bar{s} h(x)\right)\cap B\left(x,\rho\right)\subset X\left(x,2h(x),V^\perp, s\right) \cup X\left(y,2h(x),V^\perp, s\right).
\end{equation}
Let $z\in C\cap N\!\left(xV^\perp,\bar{s} h(x)\right)\cap B\left(x,\rho\right)$. If $z\in X\!\left(x,V^\perp,\bar{s}\right)$ then $\|x^{-1}z\|\le h(x)$, otherwise
\[
\|x^{-1}z\|<\frac{1}{\bar{s}}d\left(x^{-1}z,V^\perp\right)\le h(x). 
\]
Hence we always have that $\|x^{-1}z\|\le h(x)$. We also have that,
\[
\|y^{-1}z\|\le \|x^{-1}z\|+\|x^{-1}y\|\le 2h(x).
\]
It follows that $z\in B(x,2h(x))\cap B(y,2h(x))$. 

\noindent
Let $v\in V^\perp$ such that $d(x^{-1}z,v):=d(x^{-1}z,V^\perp)$. One easily obtains that \[
\|v\|\le 2\|x^{-1}z\|\le 2h(x).\]
We have
\begin{eqnarray*}
d(y^{-1}z,V^\perp)&=&d(y^{-1}x\cdot x^{-1}z,V^\perp)\le d(x^{-1}z,v) + d(y^{-1}x\cdot v,V^\perp)\\
&\le& d(x^{-1}z,V^\perp) + \underline{d(y^{-1}x\cdot v,V^\perp)}
\end{eqnarray*}

\noindent
We evaluate the underlined: Denote $\tilde{x}=y^{-1}x\cdot v$. Then
\[
\|\tilde{x}\|\le \|x^{-1}y\| + \|v\|\le 3h(x).
\]
We have that
\[
d(\tilde{x}^{-1}, V^{\perp})\le d(x^{-1}y,V^\perp)\le \bar{s}\|x^{-1}y\|\le \bar{s}h(x)
\]
Applying Lemma \ref{ineq:cone_inversion} with $\alpha=\sqrt{3}, \, \beta=1$ and $M=h(x)$
we obtain 
\[
d(y^{-1}x\cdot v,V^\perp)=d(\tilde{x},V^\perp)\le 7\bar{s}^{1/2}h(x)
\]
\noindent
Now, using $h(x)<\frac{10}{9}\|x^{-1}y\|$ and $\Bar{s}=s^2/100$, we obtain
\begin{eqnarray*}
d(x^{-1}z,V^\perp)+d\left(y^{-1}z,V^\perp\right) &\le& 2d\left(x^{-1}z,V^\perp\right) + d(y^{-1}x\cdot v,V^\perp)\\
&\le& 2\bar{s}h(x) + 7\bar{s}^{1/2}h(x)\le 9\bar{s}^{1/2}h(x)\\
&<& 10\bar{s}^{1/2}\|x^{-1}y\|=s\|x^{-1}y\|\\
&\le& s\|x^{-1}z\| + s\|y^{-1}z\|.
\end{eqnarray*}
Hence, the inclusion \eqref{eq:2cones} follows. Using \eqref{eq.assumpt.1} at both $x$ and $y$ in $E$, we obtain
\begin{eqnarray}\label{eq:cyl-estimate}\nonumber
\H^{k_m}\!\!\left(C\cap N(xV^\perp,\bar{s} h(x))\cap B\left(x,\rho\right)\right)\!\!&\le &\!\! 2\lambda(2h(x)s)^{k_m}\\
&=&\!\! 2^{k_m+1}\!\cdot\! 2000^{4k_m}\lambda\!\cdot\! s^{-7k_m}(\bar{s}^4_1 h(x))^{k_m}.
\end{eqnarray}
\noindent
\textbf{Step 4: Concluding density estimate.}
Observe from \eqref{incl:covering}, \eqref{incl:tilde_covering} and \eqref{Claim:enlarged_cylinder} that
\[
C\subset \bigcup_{x\in D}C\cap N\left(xV^\perp,\bar{s}h(x)\right)\cap B(x,\rho).
\]
We obtain after summing the estimate in \eqref{eq:cyl-estimate} over all $x\in D$, using \eqref{ineq:sum-estimate} and taking into account that $\H^{k_m}(F\setminus C)=0$,
\[
\cH^{k_m} \left(E\cap B \left(x_0,\rho \right)\right) \leq (2^{11k_m+1}\cdot 2000^{4k_m} s^{-7k_m})\lambda\rho^{k_m}.
\]
The upper density estimate immediately follows.
\end{proof}

\section{Proof of the main result}\label{sec:main}
We now turn to the proof of Theorem~\ref{Thm:Main}. 
The argument follows a standard scheme, based on the decay estimate established in Section~\ref{sec:density}. 
Similar methods have been employed in the Euclidean setting \cite{Mat95} and more recently in homogeneous groups \cite{IMM20}. We first recall the following measurability result proved in \cite{MSSC10}*{Proposition 3.9}.

\begin{proposition}\label{prop:uniqaptan}
Let $E \subset \bH^n$ be $\H^{k_m}$-measurable with $\H^{k_m}(E)<\infty$, and let $A$ be the set of points of $E$ for which there is an approximate tangent subgroup of dimension $k$. Then the following properties hold:
\begin{enumerate}[label=(\alph*)]
\item $A$ is $\H^{k_m}$-measurable;
\item $E$ has a unique approximate tangent subgroup $V_p$ at $\H^{k_m}$-a.e. $p \in A$;
\item the mapping $A \ni p \mapsto V_p \in \cG(\bH^n,k)$ is measurable.
\end{enumerate}
\end{proposition}

\vspace{1mm}
\begin{proof}[Proof of Theorem \ref{Thm:Main}]
By compactness, cover $\cG(\bH^n,k)$ with finitely many balls 
\[
B(V,1/3):=\{W\in\cG(\bH^n,k): \rho(V,W)<1/3\}.
\] 
Since $\cH^{k_m}\res E$ is locally finite, it suffices to consider the purely $k$-unrectifiable part $E_{\mathrm{pu}}$ of $E$ and prove that, for a fixed $V$,
\[
Z := \left\{ p \in E_{\mathrm{pu}} : {\textrm ap}Tan_\bH^k(E,p) = T_p \in B(V,1/3) \right\}
\]
is $\cH^{k_m}$-null.

Fix $p\in Z$ and let $s_0\in (0,c/3)$, where $c$ is the constant from Proposition~\ref{prop:proj-estimates}. We claim that
\begin{equation}\label{eq:cone-incl}
X(p,V^\perp,s_0)\subset \bH^n \setminus X(p,T_p,s_0).
\end{equation}
Indeed, if $q\in X(p,V^\perp,s_0)$, then by Proposition~\ref{prop:proj-estimates},
\[
d(p,q)\ge \frac{1}{s_0}d(p^{-1}q,V^\perp)\ge \frac{c}{s_0}\|\pi_{V^\perp}^{-1}(p^{-1}q)\cdot(p^{-1}q)\|
\ge \frac{c}{s_0}d(p,q)-\frac{c}{s_0}\|\pi_{V^\perp}(p^{-1}q)\|,
\]
which yields
\begin{equation}\label{eq:eps-range}
\|\pi_{V^\perp}(p^{-1}q)\|\ge \Bigl(1-\tfrac{s_0}{c}\Bigr)d(p,q).
\end{equation}
Since $\rho(V,T_p)=\rho(V^\perp,T_p^\perp)$, we have $T_p^\perp \in B(V^\perp,1/3)$, and hence
\[
\|\pi_{V^\perp}(p^{-1}q)\| - \|\pi_{T_p^\perp}(p^{-1}q)\| 
\le \|\pi_{T_p^\perp}^{-1}(p^{-1}q)\cdot \pi_{V^\perp}(p^{-1}q)\| 
\le \tfrac{1}{3}\|p^{-1}q\|.
\]
Combining with \eqref{eq:eps-range} gives
\[
\|\pi_{T_p^\perp}(p^{-1}q)\|\ge \Bigl(\tfrac{2}{3}-\tfrac{s_0}{c}\Bigr)d(p,q).
\]
Another application of Proposition~\ref{prop:proj-estimates} yields
\[
d(p^{-1}q,T_p)\ge c\Bigl(\tfrac{2}{3}-\tfrac{s_0}{c}\Bigr)d(p,q) > s_0 d(p,q),
\]
so $q\notin X(p,T_p,s_0)$, proving \eqref{eq:cone-incl}.

Now the approximate tangent subgroup condition \eqref{eq:TanConeCond} and \eqref{eq:cone-incl} imply that for any $\lambda>0$ there exists $\delta_{p,\lambda}>0$ such that
\begin{equation}\label{eq:estimFederer}
\cH^{k_m}\!\left(E\cap B(p,r)\cap X(p,V^\perp,s_0)\right)
\le \lambda s_0^{k_m} r^{k_m}, \qquad 0<r\le \delta_{p,\lambda}.
\end{equation}

Fix $\lambda>0$. For each $j\geq 1$, we define
\[
Z_j:=\left\{q\in Z:
\mathcal H^{k_m}\big(E\cap B(q,r)\cap
X(q,V^\perp,s_0)\big)
\leq \lambda s_0^{k_m}r^{k_m}
\text{ for every }0<r<3^{-j}\right\}.
\]
By \eqref{eq:estimFederer}, for every $q\in Z$ there exists $\delta_{q,\lambda}>0$ such that the above estimate holds for $0<r\leq\delta_{q,\lambda}$. Hence
\[
Z=\bigcup_{j=1}^{\infty}Z_j.
\] 
For fixed $j$, using Proposition \ref{prop:uniqaptan} it is not difficult to show that the set $Z_j$ is $\mathcal H^{k_m}$-measurable.
Since $\mathcal H^{k_m}(Z_j)<\infty$, inner regularity yields compact sets
\[
K_{j,\ell}\subset Z_j,\qquad \ell\geq1,
\]
such that
\[
\mathcal H^{k_m}\left(
Z_j\setminus\bigcup_{\ell=1}^{\infty}K_{j,\ell}
\right)=0.
\]
For every $p\in K_{j,\ell}$ and every $0<r\leq 3^{-(j+1)}$, the defining estimate for $Z_j$ gives
\[
\mathcal H^{k_m}\big(E\cap B(p,r)\cap
X(p,V^\perp,s_0)\big)
\leq \lambda s_0^{k_m}r^{k_m}.
\]
Thus all the hypotheses of Theorem~\ref{thm:density-estimate} are satisfied on $K_{j,\ell}$. Consequently,
\[
\Theta^{*k_m}(K_{j,\ell},p)
\leq C(k_m,s_0)\lambda
\qquad\text{for every }p\in K_{j,\ell}.
\]
where $C(k_m,s_0)=2^{11k_m+1}\cdot 2000^{4k_m}s_0^{-7k_m}$. Choosing $\lambda$ sufficiently small and applying Lemma~\ref{lemma:densities} gives \[
\mathcal H^{k_m}(K_{j,\ell})=0.
\]
Since $Z$ is covered, up to an $\mathcal H^{k_m}$-null set, by the compact sets
$K_{j,\ell}$, it follows that
$\mathcal H^{k_m}(Z)=0$, completing the proof.
\end{proof}

\section{Application}\label{sec:applications}
A recent theorem on higher-order rectifiability in Heisenberg groups \cite{IM21}*{Theorem~1.1} provides sufficient conditions for low-codimension subsets to be \textit{${C}^{1,\alpha}$-rectifiable}. 
These conditions include, in particular, a positive lower density assumption. 
As an application of Theorem~\ref{Thm:Main}, we show that this density assumption can be removed, thereby yielding a strengthened version of the result. 
Before stating the precise formulation, we recall a few essential notions.

\subsection{\texorpdfstring{$C^{1,\alpha}$}-rectifiable sets in low codimensions}

\begin{definition}\label{Def:CH-regular}
Let $f\in \mathbf{C}_{H}^{1}(\mathcal{U})$. We define the horizontal gradient of $f$ as
\[
\nabla_{H} f:=\left(X_{1} f, \ldots, X_{n} f, Y_{1} f, \ldots, Y_{n} f\right).
\]
Suppose $\nabla_H f$ is $\alpha$-H\"{o}lder continuous with respect to the homogeneous norm, i.e.
\[
{\sup_{p,q\in \Omega, \, p\ne q}}\frac{\|\nabla_H f(p)-\nabla_H f(q)\|_{\R^{2n}}}{\|p^{-1}q\|^\alpha} < \infty
\]
for some $\alpha \in (0,1]$, then we say that $f\in \mathbf{C}_{H}^{1,\alpha}(\Omega)$.
\end{definition}

\begin{definition}
Let $n+1\le k \le 2n$. A set $S \subset \bH^n$ is a $k$-dimensional $(\mathbf{C}_H^{1,\alpha},\bH)$-regular surface if for any $p \in S$ there are $\cU \subseteq \bH^n$ open and $f \in [\mathbf{C}_H^{1,\alpha}(\cU)]^{2n+1-k}$ satisfying
\begin{enumerate}[label=(\alph*), itemsep=1ex]
\item $d_H f_q$ is surjective at all $q \in \cU$;
\item $S \cap \cU = \{ q \in \cU \: : \: f(q) = 0 \}$.
\end{enumerate}
\end{definition}

\begin{definition} \label{def:rec1a}
A measurable set $E \subset \bH^n$ is {} $C^{1,\alpha}$-rectifiable if there are $k$-dimensional $(\mathbf{C}_H^{1,\alpha},\bH)$-regular surfaces $S_i$, with $i \in \N$, such that
\[
\cH^{k_m} \left(E\sm\bigcup_{i\in\N}S_i\right)=0,
\]
where $k_m=k$ if {} $1\le k \le n$ and $k_m=k+1$ if {} $n+1\le k \le 2n$.
\end{definition}

\begin{definition}\label{def.objects}
Fix $\alpha \in (0,1]$ and $\lambda >0$. The {\itshape $\alpha$-paraboloid} centered at $x\in \bH^n$ with base $V \in \cG(\bH^n)$ and parameter $\lambda$ is defined as
\[
Q_\alpha(x,V,\lambda) := \left\{ y \in \bH^n \ : \ d(x^{-1}y,V) \le \lambda d(x,y)^{1+\alpha} \right\}.
\]
\end{definition}

\begin{theorem} \label{thm.1.1} 
Fix $\alpha \in (0,1]$ and $n<k\le 2n$. Let $E \subset \bH^n$ be a $\H^{k_m}$-measurable set with $\H^{k_m}(E) < \infty$ such that for $\H^{k_m}$-a.e. $p \in E$ there are $V_p \in \cG(\bH^n,k)$ and $\lambda > 0$ such that
\beq\label{eq.paraboloidassumption}
\lim_{r \to 0^+} \frac{1}{r^{k_m}} \H^{k_m} \left( E \cap B(p,r) \setminus Q_\alpha(p,V_p,\lambda) \right) = 0.
\eeq
Then $E$ is $C^{1,\alpha}$-rectifiable in the sense of Definition \ref{def:rec1a}.
\end{theorem}

\begin{proof}
An immediate consequence of the local inclusion of paraboloids in cones at small scales gives that the approximate tangent paraboloid condition \eqref{eq.paraboloidassumption} implies the approximate tangent subgroup condition \eqref{eq:TanConeCond}. 
It follows from Theorem \ref{Thm:Main} that $E$ is $(k,\bH)-$rectifiable. We therefore have from \cite{MSSC10}*{Theorem 3.15} that for $\H^{k_m}$-a.e. $p \in E$ it holds that
\[
\Theta_\ast^{k_m}(E,p) > 0.
\]
With the positive lower density condition thus established, the rest of the proof follows as in \cite{IM21}*{Theorem 1.1}.
\end{proof}
\vspace{10mm}

\noindent
\textbf{Data Availability} This work has no associated data.

\section{Declarations}
\noindent
\textbf{Conflict of interest} The author states that there is no Conflict of
interest.

\bibliographystyle{plain}
\bibliography{bibtex}

@article{Mat2022, title={Parabolic rectifiability, tangent planes and tangent measures}, volume={47}, url={https://afm.journal.fi/article/view/119821}, DOI={10.54330/afm.119821}, number={2}, journal={Annales Fennici Mathematici}, author={Mattila, Pertti}, year={2022}, pages={855–884} }

@article{IM21,
url = {https://doi.org/10.1515/agms-2023-0105},
title = {${C}^{1,\alpha}$-rectifiability in low codimension in Heisenberg groups},
author = {Kennedy Obinna Idu and Francesco Paolo Maiale},
pages = {20230105},
volume = {12},
number = {1},
journal = {Analysis and Geometry in Metric Spaces},
doi = {doi:10.1515/agms-2023-0105},
year = {2024},
}

@book{Fed69,
	Author = {H. Federer},
	Publisher = {Springer},
	Title = {Geometric Measure Theory},
	Year = {1969},}

@article{FSSC6,
	Author = {Franchi, Bruno and Serapioni, Raul and Serra Cassano, Francesco},
	Doi = {10.1016/j.aim.2006.07.015},
	Fjournal = {Advances in Mathematics},
	Issn = {0001-8708},
	Journal = {Adv. Math.},
	Mrclass = {49Q15 (22E25 28A75 28A78)},
	Mrnumber = {2313532},
	Mrreviewer = {Oleg G. Okunev},
	Number = {1},
	Pages = {152--203},
	Title = {Regular submanifolds, graphs and area formula in {H}eisenberg groups},
	Url = {https://doi.org/10.1016/j.aim.2006.07.015},
	Volume = {211},
	Year = {2007}}

@article{FranchiSerapioni2016IntrLip,
	Author = {Franchi, Bruno and Serapioni, Raul Paolo},
	Doi = {10.1007/s12220-015-9615-5},
	Fjournal = {Journal of Geometric Analysis},
	Issn = {1050-6926},
	Journal = {J. Geom. Anal.},
	Mrclass = {49Q15 (22E25 53C17 58C20 58C35)},
	Mrnumber = {3511465},
	Mrreviewer = {Jingzhi Tie},
	Number = {3},
	Pages = {1946--1994},
	Title = {Intrinsic {L}ipschitz graphs within {C}arnot groups},
	Url = {https://doi.org/10.1007/s12220-015-9615-5},
	Volume = {26},
	Year = {2016}}

@book {Mat_Surv2023,
    AUTHOR = {Mattila, Pertti},
     TITLE = {Rectifiability---a survey},
    SERIES = {London Mathematical Society Lecture Note Series},
    VOLUME = {483},
 PUBLISHER = {Cambridge University Press, Cambridge},
      YEAR = {2023},
     PAGES = {vii+172},
      ISBN = {978-1-009-28808-8},
   MRCLASS = {28-02 (28A75 30L99 43A05 49Q15 49Q20 53C17 53C23)},
}

@article{MSSC10,
	Author = {Mattila, P. and Serapioni, R. and Serra Cassano, F.},
	Fjournal = {Annali della Scuola Normale Superiore di Pisa. Classe di Scienze. Serie V},
	Issn = {0391-173X},
	Journal = {Ann. Sc. Norm. Super. Pisa Cl. Sci. (5)},
	Number = {4},
	Pages = {687--723},
	Title = {Characterizations of intrinsic rectifiability in {H}eisenberg groups},
	Volume = {9},
	Year = {2010}}

@article {Vittone2022,
    AUTHOR = {Vittone, Davide},
     TITLE = {Lipschitz graphs and currents in {H}eisenberg groups},
   JOURNAL = {Forum Math. Sigma},
  FJOURNAL = {Forum of Mathematics. Sigma},
    VOLUME = {10},
      YEAR = {2022},
     PAGES = {Paper No. e6, 104},
      ISSN = {2050-5094},
   MRCLASS = {49Q15 (22E30 26A16 53C17)},
}

@book{BLU2007stratified,
	Author = {Bonfiglioli, A. and Lanconelli, E. and Uguzzoni, F.},
	Isbn = {978-3-540-71896-3; 3-540-71896-6},
	Mrclass = {22E30 (31C45 35-02 35H10 43A80)},
	Mrnumber = {2363343},
	Mrreviewer = {Maria Stella Fanciullo},
	Pages = {xxvi+800},
	Publisher = {Springer, Berlin},
	Series = {Springer Monographs in Mathematics},
	Title = {Stratified {L}ie groups and potential theory for their sub-{L}aplacians},
	Year = {2007}}

@book{Mat95,
	Address = {Cambridge},
	Author = {Mattila, Pertti},
	Isbn = {0-521-46576-1; 0-521-65595-1},
	Pages = {xii+343},
	Publisher = {Cambridge University Press},
	Series = {Cambridge Studies in Advanced Mathematics},
	Title = {Geometry of sets and measures in {E}uclidean spaces},
	Volume = {44},
	Year = {1995}}

@article{IMM20,
title={Characterizations of k-rectifiability in homogeneous groups},
  author={Idu, Kennedy Obinna and Magnani, Valentino and Maiale, Francesco Paolo},
  journal={J. Math. Anal. Appl},
  volume={500},
  number={2},
  pages={125120},
  year={2021},
  publisher={Elsevier}}

\end{document}